\documentclass[12pt]{amsart}
\usepackage{latexsym}
\usepackage{amsfonts,amsmath,amssymb,indentfirst}
\usepackage[pdftex]{hyperref}
\newcommand{\beq}{\begin{equation}}
\newcommand{\eeq}{\end{equation}}
\newcommand{\bdism}{\begin{displaymath}}
\newcommand{\edism}{\end{displaymath}}

\newcommand{\setR}{\mathbb R}
\newcommand{\setC}{\mathbb C}

\newtheorem{main}{Theorem}

\newtheorem{others}{Theorem}
\newtheorem{theorem}{Theorem}[section]
\newtheorem{proposition}[theorem]{Proposition}
\newtheorem{corollary}[theorem]{Corollary}
\newtheorem{lemma}[theorem]{Lemma}

\newtheorem{definition}{Definition}

\author[Di Cerbo]{Luca Fabrizio Di Cerbo}

\title{Seiberg-Witten equations on surfaces of logarithmic general type}
\begin{document}
\maketitle

\begin{abstract}
We study the Seiberg-Witten equations on surfaces of logarithmic
general type. First, we show how to construct irreducible solutions
of the Seiberg-Witten equations for any metric which is
``asymptotic'' to a Poincar\'e type metric at infinity. Then we
compute a lower bound for the $L^{2}$-norm of scalar curvature on
these spaces and give non-existence results for Einstein metrics on
blow-ups.
\end{abstract}

\tableofcontents

\section{Introduction}

This paper is concerned with \emph{Seiberg-Witten equations} on
non-compact complex $4$-manifolds. More precisely, we are interested
in complex manifolds obtained from a smooth projective surface
$\overline{M}$ by removing a smooth, reduced, not necessarily
connected, divisor $\Sigma$. Recall that in this case a reduced
divisor is simply a collection of complex submanifolds or otherwise
said smooth holomorphic curves. We denote by $M$ the complex
manifold obtained from $\overline{M}$ by removing $\Sigma$.

The main problem with Seiberg-Witten theory on non-compact manifold
is the lack of a satisfactory existence theory. For a readable
account of what is known in the compact case one may refer to
\cite{Morgan1}. Following a beautiful approach due to Biquard
\cite{Biquard}, we solve the SW equations on $M$ by working on the
compactification $\overline{M}$. More precisely, we produce an
\emph{irreducible} solution of the unperturbed SW equations on $M$
as limit of solutions of the perturbed SW equations on
$\overline{M}$. From the metric point of view, starting with $(M,g)$
where $g$ is assumed to be asymptotic to a Poincar\'e type metric at
infinity, one has to construct a sequence $(\overline{M},g_{j})$ of
metric compactifications that approximate $(M,g)$ as $j$ diverges.
The irreducible solution of the SW equations on $(M,g)$ is then
constructed by a bootstrap argument with the solutions of the SW
equations on $(\overline{M},g_{j})$ with suitably constructed
perturbations.

Here is an outline of the paper. Section \ref{compactification}
describes explicitly the metric compactifications $(\overline{M},
g_{j})$. These metrics are completely analogous to the one used by
Rollin and Biquard in \cite{Rollin} and \cite{Biquard}. Furthermore,
few results concerning the Riemannian geometry of the spaces
$(\overline{M}, g_{j})$ are given.

In Section \ref{l2lebrun} we review some classical facts about the
$L^{2}$-cohomology of complete manifolds. Moreover, we recall a
fundamental result of Zucker \cite{Zucker} concerning the
$L^{2}$-cohomology in the Poincar\'e metric. Finally, we formulate
the $L^{2}$-analogue of LeBrun's scalar curvature estimate
\cite{LeBrun2}.

Sections \ref{analisi1} and \ref{analisi2} contain the uniform
Poincar\'e inequalities on functions and $1$-forms needed for the
main analytical argument. Moreover the convergence, as $j$ goes to
infinity, of the harmonic forms on $(\overline{M}, g_{j})$ is
studied in detail.

In Section \ref{bootstrap} the bootstrap argument is worked out. The
existence result so obtained is summarized in Theorem
\ref{convergence}.

In Section \ref{application}, Theorem \ref{convergence} is applied
to derive several geometrical consequences. First, we give a lower
bound for the Riemannian functional $\int s_{g}^{2}d\mu_{g}$ on $M$,
where by $s$ we denote the scalar curvature. Second, an obstruction
to the existence of Einstein metrics on blow-ups of $M$ is given.
These results are summarized in Theorem \ref{minima} and Theorem
\ref{ostruzione}. These theorems are the finite volume
generalization of some well-known results of LeBrun for closed four
manifolds, see for example \cite{LeBrun5}.

\section{Logarithmic Kodaira dimension}\label{logarithmic}

In this section, we review some of the complex geometry needed in
this paper. In algebraic geometry \cite{Iitaka}, the object of study
is usually not an open complex manifold $M$ but rather a \emph{pair}
$(\overline{M}, \Sigma)$ consisting of an algebraic variety
$\overline{M}$, and a reduced simple normal crossing divisor
$\Sigma$. The obvious requirement is that $M$ is biholomorphic to
$\overline{M}\backslash\Sigma$. The variety $\overline{M}$ is called
the \emph{smooth completion} or simply the \emph{compactification}
of $M$. The divisor $\Sigma$ is referred as the \emph{boundary} of
the smooth completion.

In this paper, unless otherwise stated, $\overline{M}$ will be a
smooth, projective, surface and $\Sigma$ a reduced, smooth, divisor
without rational components. The general case including rational
components and simple normal crossing boundary divisors will be
treated elsewhere.

Motivated by the well-known properties of the minimal model in
complex dimension two \cite{Van de Ven}, we introduce a notion of
minimality for a pair $(\overline{M}, \Sigma)$. For more details see
\cite{Iitaka}.

\begin{definition}
Let $(\overline{M},\Sigma)$ be as above. The pair is called minimal
if $\overline{M}$ does not contain an exceptional curve $E$ of the
first kind such that $E\cdot\Sigma\leq 1$.
\end{definition}

Recall that an exceptional curve of the first kind is simply a
rational curve $E$ with self-intersection minus one, see
\cite{Griffiths} and \cite{Van de Ven}.

Given a pair $(\overline{M}, \Sigma)$, consider the
\emph{logarithmic canonical bundle}
$\mathcal{L}=K_{\overline{M}}+\Sigma$. Given any integer $m$, define
the \emph{logarithmic plurigenera} as
$\overline{P}_{m}=\textrm{dim}H^{0}(\overline{M},\mathcal{O}(m\mathcal{L}))$.
If $\overline{P}_{m}>0$, we define the $m$-th \emph{logarithmic
canonical map} $\Phi_{m\mathcal{L}}$ of the pair
$(\overline{M},\Sigma)$ by
\begin{align}\notag
\Phi_{m\mathcal{L}}(x)=[s_{1}(x), ...,s_{N}(x)],
\end{align}
where $s_{1}, ...,s_{N}$ is a basis for the vector space
$H^{0}(\overline{M},\mathcal{O}(m\mathcal{L}))$. Here the point $x$
has to chosen not in the \emph{base locus} $B_{m}$ of the linear
series $|m\mathcal{L}|$. Concretely, $B_{m}$ is simply the set of
points where all the sections of $m\mathcal{L}$ vanish. At this
point one introduces the notion of \emph{logarithmic Kodaira
dimension} exactly as in the closed smooth case. More precisely, the
logarithmic Kodaira dimension is defined to be the maximal dimension
of the image of the logarithmic canonical maps. Conventionally, we
assign dimension $-\infty$ to pairs for which the logarithmic
plurigera are identically zero. We denote this numerical invariant
by $\overline{k}(M)$. A pair $(\overline{M}, \Sigma)$ is said to be
of \emph{logarithmic general type} if it has maximal logarithmic
Kodaira dimension.

Next, we have to recall the concept of \emph{numerical
effectiveness} for line bundles on surfaces. For an extensive
treatment of this circle of ideas in complex algebraic geometry we
refer to the book \cite{Lazarsfeld}.

\begin{definition}
Let $\overline{M}$ be a smooth projective surface. A line bundle $L$
on $\overline{M}$ is said to be numerical effective, or simply nef,
if
\begin{align}\notag
\int_{D}c_{1}(L)\geq 0
\end{align}
for any irreducible curve $D$ in $\overline{M}$.
\end{definition}

We are now ready to state and prove the following result:

\begin{lemma}\label{numerical}
Let $(\overline{M},\Sigma)$ be a minimal pair such that $\Sigma$ is
a reduced, smooth divisor without rational curves. If
$\overline{k}(M)\geq 0$, then $\mathcal{L}$ is numerically
effective.
\end{lemma}

\begin{proof}
Let $E$ be an irreducible components of the boundary divisor
$\Sigma$. Since $\Sigma$ does not contain rational components, by
the adjunction formula we conclude that $\mathcal{L}\cdot E\geq 0$.
Thus, it remains to check that $\mathcal{L}\cdot E\geq 0$ for any
irreducible curve $E$ which is not a component of $\Sigma$. Let us
proceed by contradiction. Suppose $\mathcal{L}\cdot E<0$. Since
$\overline{k}(M)\geq 0$, there exists $m$ such that $m\mathcal{L}$
is effective. We then conclude that $E$ must be a component of
$m\mathcal{L}$ with negative self-intersection. On the other hand
\begin{align}\notag
K_{\overline{M}}\cdot E+\Sigma\cdot E<0 \quad \Longrightarrow \quad
K_{\overline{M}}\cdot E<0.
\end{align}
It then follows that $E\simeq\setC P^{1}$ and $E^{2}=-1$. As a
result we must have $\Sigma\cdot E=0$. This contradicts the
minimality of the pair.
\end{proof}

We can now concentrate on the case when the pair $(\overline{M},
\Sigma)$ is of logarithmic general type.

\begin{proposition}\label{exceptional}
Let $(\overline{M},\Sigma)$ be a minimal pair such that $\Sigma$ is
a reduced, smooth divisor without rational curves. If
$\overline{k}(M)=2$, then $\mathcal{L}^{2}>0$.
\end{proposition}
\begin{proof}
By Lemma \ref{numerical}, the line bundle $\mathcal{L}$ is nef.
Since the Kodaira dimension of $\mathcal{L}$ is nonnegative, by
Serre duality we conclude that $H^{2}(\overline{M},
\mathcal{O}(m\mathcal{L}))$ vanishes for all positive integers $m$.
A Hirzebruch-Riemann-Roch computation shows that
\begin{align}\notag
\overline{P}_{m}=h^{1}(m\mathcal{L})+\frac{m(m-1)}{2}\mathcal{L}^{2}+\frac{m}{2}\mathcal{L}\cdot\Sigma+\chi_{\mathcal{O}}
\end{align}
where by $\chi_{\mathcal{O}}$ we denote the holomorphic Euler
characteristic of $\overline{M}$. Now, by an important
generalization of the classical Kodaira-Nakano vanishing theorem due
to Kawamata and Viehweg \cite{Lazarsfeld}, we know that
$h^{1}(m\mathcal{L})=0$ for any positive $m$. Since
$\overline{P}_{m}$ grows quadratically in $m$, we conclude that
$\mathcal{L}^{2}>0$.
\end{proof}

We conclude this section with a proposition regarding the
\emph{ampleness} properties of $\mathcal{L}$. This proposition will
be of some importance in Section \ref{application}.

\begin{proposition}\label{ample}
Let $(\overline{M},\Sigma)$ be a minimal pair such that $\Sigma$ is
smooth and reduced. The log-canonical bundle $\mathcal{L}$ is ample
iff $\overline{k}(M)=2$, $\Sigma$ does not contain rational and/or
elliptic components, and there are no interior rational
$(-2)$-curves.
\end{proposition}
\begin{proof}
If $\mathcal{L}$ is ample, we clearly have $\overline{k}(M)=2$.
Moreover by the adjunction formula $\Sigma$ cannot contain rational
and/or elliptic components, and there are no interior rational
$(-2)$-curves, i.e., $(-2)$-curves not intersecting the boundary
divisor $\Sigma$. Assume now $(\overline{M}, \Sigma)$ to be
log-general and with $\Sigma$ not containing rational and elliptic
curves. By Proposition \ref{exceptional} we know that
$\mathcal{L}^{2}>0$. By the Hodge index theorem any divisor $E$ such
that $\mathcal{L}\cdot E=0$ must have negative self-intersection.
Since by assumption the pair is minimal and there are no interior
$(-2)$-curves, the proposition is now a consequence of Nakai's
criterion for ampleness of divisors on surfaces, see \cite{Van de
Ven}.

\end{proof}

\section{Poincar\'e metrics}

Let ($\overline{M}, \omega_{0})$ be a smooth algebraic surface
equipped with a K\"ahler metric. Let $\Sigma$ be a smooth, not
necessarily connected, reduced divisor on $\overline{M}$. Choose an
Hermitian metric $\|\cdot\|$ on
$\mathcal{O}_{\overline{M}}(\Sigma)$. Let $s\in H^{0}(\overline{M},
\mathcal{O}(\Sigma))$ be a defining section for the divisor
$\Sigma$. On $M$, for $T>0$ big enough
\begin{align}\label{Poincare}
\omega=T\omega_{0}-\sqrt{-1}
\partial\overline{\partial}\log\log^{2}\|s\|^{2}
\end{align}
defines a complete cuspidal K\"ahler metric. On each cusp this
metric asymptotically looks like a product of the Poincar\'e
punctured metric on the disk and the divisor. The number of cusp is
clearly in one to one correspondence with the number of connected
components of $\Sigma$ which we denote by $\Sigma_{i}$. More
precisely, one can easily show that on each cusp the Riemannian
metric associated to (\ref{Poincare}) is given by
\begin{align}\label{model}
g=dt^{2}+e^{-2t}\eta^{2}_{i}+p^{*}g_{\Sigma_{i}}+O(e^{-t})
\end{align}
where $\eta_{i}$ is a connection $1$-form for the normal bundle of
$\Sigma_{i}$, $g_{\Sigma_{i}}$ is a \emph{fixed} metric on
$\Sigma_{i}$, and where the perturbation $O(e^{-t})$ is understood
at any order for $g$. Note that the metric $g_{\Sigma_{i}}$ is
simply obtained by restricting
\begin{align}\notag
\hat{\omega}=T\omega_{0}-\frac{2\sqrt{-1}}{\log\|s\|^{2}}
\partial\overline{\partial}\log\|s\|^{2}
\end{align}
to the holomorphic curve $\Sigma_{i}$. Finally, we denote by
\begin{align}\notag
p: N_{\Sigma_{i}}\longrightarrow\Sigma_{i}
\end{align}
the $S^{1}$ bundle associated to the normal bundle of $\Sigma_{i}$
in $\overline{M}$. From the Riemannian geometry point of view we
have
\begin{align}\notag
g=dt^{2}+g_{t}+O(e^{-t})
\end{align}
where $g_{t}$ is a $1$-parameter families of metrics on
$N_{\Sigma_{i}}$ which are \emph{Riemannian submersions} with
respect to $p_{\eta_{i}}$ and $g_{\Sigma_{i}}$. These metrics are
\emph{volume collapsing} in the direction transverse to the divisor
$\Sigma$ only.

\section{$L^{2}$-cohomology and Seiberg-Witten estimates}\label{l2lebrun}

Let us begin with a review of some facts about $L^{2}$-cohomology
and its relation to the space of $L^{2}$-harmonic forms. The
interested reader is referred to \cite{Anderson1} for a nice survey
in this area. Given a orientable noncompact manifold $(M,g)$ we
have, when the differential $d$ is restricted to an appropriate
dense subset, a Hilbert complex
\begin{align}\notag
...\longrightarrow L^{2}\Omega^{k-1}_{g}(M)\longrightarrow
L^{2}\Omega^{k}_{g}(M)\longrightarrow
L^{2}\Omega^{k+1}_{g}(M)\longrightarrow...
\end{align}
where the inner products on the exterior bundles are induced by $g$.
Define the maximal domain of $d$, at the $k$-th level, to be
\begin{align}\notag
Dom^{k}(d)=\{\alpha\in L^{2}\Omega^{k}_{g}(M), d\alpha\in
L^{2}\Omega^{k+1}_{g}(M)\}
\end{align}
where $d\alpha\in L^{2}\Omega^{k+1}_{g}(M)$ is to be understood in
the distributional sense. The (reduced) $L^{2}$-cohomology groups
are then defined to be
\begin{align}\notag
H^{k}_{2}(M)=Z^{k}_{2}(M)/\overline{d Dom^{k-1}(d)},
\end{align}
where
\begin{align}\notag
Z^{k}_{2}(M)=\{\alpha\in L^{2}\Omega^{k}_{g}(M), d\alpha=0\}.
\end{align}
On $(M,g)$ there is a Hodge-Kodaira decomposition
\begin{align}\notag
L^{2}\Omega^{k}_{g}(M)=\mathcal{H}^{k}_{g}(M)\oplus\overline{dC^{\infty}_{c}\Omega^{k-1}}\oplus\overline{d^{*}C^{\infty}_{c}\Omega^{k+1}},
\end{align}
where
\begin{align}\notag
\mathcal{H}^{k}_{g}(M)=\{\alpha\in L^{2}\Omega^{k}_{g}(M),
d\alpha=0, d^{*}\alpha=0\}.
\end{align}
Moreover, if we assume $(M,g)$ to be complete the maximal and
minimal domain of $d$ coincide. In other words
\begin{align}\notag
\overline{d Dom^{k-1}(d)}=\overline{dC^{\infty}_{c}\Omega^{k-1}},
\end{align}
which implies
\begin{align}\notag
H^{k}_{2}(M)=\mathcal{H}^{k}_{g}(M).
\end{align}

Summarizing, if the manifold is complete, the harmonic $L^{2}$-forms
compute the reduced $L^{2}$-cohomology. Moreover, in this case the
$L^{2}$-harmonic forms can be characterized as follows
\begin{align}\notag
\mathcal{H}^{k}_{g}(M)=\{\alpha\in L^{2}\Omega^{k}_{2}(M), (d
d^{*}+d^{*}d)\alpha=0\}.
\end{align}
Finally, the orientability of $M$ gives a duality isomorphism via
the Hodge $*$ operator
\begin{align}\notag
\mathcal{H}^{k}_{g}(M)\simeq\mathcal{H}^{n-k}_{g}(M).
\end{align}
If the manifold $M$ has dimension $4n$ it then makes sense to talk
about $L^{2}$ self-dual and anti-self-dual forms on
$L^{2}\Omega^{2n}_{g}(M)$. If $\mathcal{H}^{2n}_{g}(M)$ is finite
dimensional, the concept of $L^{2}$-signature is well defined.

Apart from the basic concepts above, what will be important for us
is the following theorem of Zucker \cite{Zucker}. This classical
result provides a complete topological interpretation of the
$L^{2}$-cohomology groups of finite volume complex manifolds
equipped with Poincar\'e type metrics. The statement presented here
is adapted from \cite{Zucker} to fit with the purpose and notation
of this work.

\begin{others}[Zucker]\label{Zucker}
Let $\overline{M}$ be a smooth algebraic manifold. Let $\Sigma$ be a
reduced, smooth, divisor. If $M$ is equipped with a metric $g$
quasi-isometric to a Poincar\'e type metric then
\begin{align}\notag
H^{*}_{2}(M)\simeq H^{*}(\overline{M};\setR).
\end{align}
\end{others}
\vskip 5mm We close this section with the $L^{2}$-analogue of some
well-known Seiberg-Witten estimates due to LeBrun \cite{LeBrun2}.
Let $(M,g)$ be a complete finite-volume 4-manifold. Let
$\mathcal{L}$ be a complex line bundle on $M$. By extending the
Chern-Weil theory for compact manifolds, we can define the
$L^{2}$-Chern class of $\mathcal{L}$. More precisely, given a
connection $A$ on $\mathcal{L}$ such that $F_{A}\in
L^{2}\Omega^{2}_{g}(M)$, we may define
\begin{align}\notag
c_{1}(\mathcal{L})=\frac{i}{2\pi}[F_{A}]_{L^{2}}
\end{align}
where with $F_{A}$ we indicate the curvature of the given
connection. It is an interesting corollary of the $L^{2}$-cohomology
theory that, on complete manifolds, such an $L^{2}$-cohomology
element is connection independent as long as we allow connections
that differ by a $1$-form in the maximal domain of the $d$ operator.
More precisely, let $A^{'}$ be a connection on $\mathcal{L}$ such
that $A^{'}=A+\alpha$ with $\alpha\in L^{2}_{1}\Omega^{1}_{g}(M)$.
We then have $F_{A^{'}}=F_{A}+d\alpha$ and therefore by the
Hodge-Kodaira decomposition we conclude that
$\frac{i}{2\pi}[F_{A}]_{L^{2}}=\frac{i}{2\pi}[F_{A^{'}}]_{L^{2}}$.
Similarly, the associated $L^{2}$-Chern number
$c^{2}_{1}(\mathcal{L})$ is also well defined.

The following lemma is an easy consequence of the Hodge-Kodaira
decomposition. For the details of the proof see \cite{luca2}.
\begin{lemma}\label{lebruno}
Given $\mathcal{L}$ and $A$ as above, we have
\begin{align}\notag
\int_{M}|F^{+}_{A}|^{2}d\mu_{g}\geq
4\pi^{2}(c^{+}_{1}(\mathcal{L}))^{2}
\end{align}
where $c^{+}_{1}(\mathcal{L})$ is the self-dual part of the
g-harmonic $L^{2}$-representative of $[c_{1}(\mathcal{L})]$.
\end{lemma}

We can now formulate the $L^{2}$-analogue of the scalar curvature
estimate discovered in \cite{LeBrun2} for compact manifolds.
\begin{theorem}\label{bochner}
Let $(M^{4},g)$ be a finite volume Riemannian manifold where $g$ is
$C^{2}$-asymptotic to a Poincar\'e metric. Let $(A,\psi)\in
L^{2}_{1}(M,g)$ be an irreducible solution of the SW equations
associated to a \emph{$Spin^{c}$} structure $\mathfrak{c}$ with
determinant line bundle $\mathcal{L}$. Then
\begin{align}\notag
\int_{M}s^{2}_{g}d\mu_{g}\geq 32\pi^{2}(c^{+}_{1}(\mathcal{L}))^{2}
\end{align}
with equality if and only if $g$ has constant negative scalar
curvature, and is K\"ahler with respect to a complex structure
compatible with $\mathfrak{c}$.
\end{theorem}
\begin{proof}
The proof is largely based on an idea of LeBrun \cite{LeBrun2}. By
using Lemma \ref{lebruno}, the proof reduces to an integration by
part using the completeness of $g$. For the analytical details one
may refer to \cite{Rollin} and \cite{luca2}.
\end{proof}

\section{The metric compactifications}\label{compactification}

Given $(M, g)$, where $g$ is a Poincar\'e metric, we want to briefly
describe a family of metric compactifications $(\overline{M},
g_{j})$. Clearly, each of the cusp end of $M$ can be closed
topologically as a manifold by adding the corresponding divisor
$\Sigma_{i}$. Recall that on each cusp end of $M$ the metric $g$ is
given, in appropriate coordinates, by
$g=dt^{2}+e^{-2t}\eta^{2}+g_{\Sigma_{i}}+O(e^{-t})$ and for $t>0$.
We can then define, on a closed tubular neighborhood of $\Sigma$ in
$\overline{M}$, a sequence of metrics $\{\tilde{g}_{j}\}$  by
\begin{align}\notag
\tilde{g}_{j}=dt^{2}+\varphi^{2}_{j}(t)\eta^{2}_{i}+g_{\Sigma_{i}}
\end{align}
where the index $i$ runs over all the connected components of
$\Sigma$ and where $\varphi_{j}(t)$ is a smooth warping function
such that:
\begin{center}
\begin{enumerate}
\item $\varphi_{j}(t)=e^{-t}$ for $t\in[0, j+1]$;
\item $\varphi_{j}(t)= T_{j}-t$ for $t\in [j+1+\epsilon, T_{j}]$.
\end{enumerate}
\end{center}
Here $\epsilon$ is a fixed number that can be chosen to be small,
and $T_{j}$ is an appropriate number bigger than $j+1+\epsilon$.
Because of the second condition above, $\tilde{g}_{j}$ is smooth for
$t$ approaching $T_{j}$. For later convenience we want to prescribe
in more details the behavior of $\varphi_{j}(t)$ in the interval
$t\in[j+1, j+1+\epsilon]$. We require that
$\partial^{2}_{t}\varphi_{j}(t)$ decreases from $e^{-(j+1)}$ to $0$
in the interval $[j+1, j+1+\delta_{j}]$ where $\delta_{j}$ is a
positive number less than $\epsilon$. Then for
$t\in[j+1+\delta_{j},\epsilon]$, we make
$\partial^{2}_{t}\varphi_{j}$ very negative in order to decrease
$\partial_{t}\varphi_{j}$ to $-1$ and smoothly glue $\varphi_{j}(t)$
to the function $T_{j}-t$. Moreover, by eventually letting the
parameters $\delta_{j}$ go to zero as $j$ goes to infinity, we
require $\frac{|\partial_{t}\varphi_{j}|}{\varphi_{j}}$ to be
increasing in the interval $[j+1, j+1+\delta_{j}]$. Finally, we
require $\frac{|\partial_{t}\varphi_{j}|}{\varphi_{j}}$ to be
bounded from above uniformly in $j$. These conditions on the warping
factors $\varphi_{j}$ are taken from \cite{Rollin}. They are
particularly useful in proving the uniform Poincar\'e inequality on
$1$-form given in section \ref{analisi1}.

With the metrics $\tilde{g_{j}}$ at our disposal, we are now ready
to describe the family of metric compactifications $(\overline{M},
g_{j})$. For later convenience, we allow the metric $g$ to be
asymptotic to a Poincar\'e type metric in the $C^{2}$-topology.
Thus, if $g$ is such a metric we set
\begin{align}\notag
g_{j}=(1-\chi_{j})g+\chi_{j}\tilde{g}_{j}
\end{align}
where ${\chi_{j}(t)}$ is a sequence of smooth increasing functions
defined on the cusps of $M$ such that $\chi_{j}(t)=0$ if $t\leq j$
and $\chi_{j}(t)=1$ if $t\geq j+1$. Note that the metrics $g_{j}$
are by construction isometric to $g$ on bigger and bigger compact
sets of $M$ as $j$ goes to infinity. In this sense we think the
sequence of Riemannian manifolds $(\overline{M}, g_{j})$ as compact
approximations of $(M, g)$.

We conclude this section with two simple propositions regarding the
volume and the scalar curvature of the Riemannian manifolds
$(\overline{M},g_{j})$.

\begin{proposition}\label{scalar curvature}
The scalar curvature of the metrics $\{g_{j}\}$ can be expressed as
\begin{align}\notag
s_{g_{j}}=s^{b}_{g_{j}}-2\chi_{j}\frac{\partial^{2}_{t}\varphi_{j}}{\varphi_{j}}
\end{align}
where $s^{b}_{g_{j}}$ is a smooth function on $\overline{M}$ that
can be bounded uniformly in $j$.
\end{proposition}

\begin{proposition}\label{volumes}
There exists a constant $K>0$ such that
\begin{align}\notag
Vol_{g_{j}}(\overline{M})\leq K
\end{align}
for any $j$.
\end{proposition}

One may refer to \cite{luca1} and \cite{Rollin} for more details.

\section{Poincar\'e inequalities and convergence of
1-forms}\label{analisi1}

We need to show that, given the sequence of metrics $\{g_{j}\}$, we
can find a uniform Poincar\'e inequality on functions. We have the
following lemma.
\begin{lemma}\label{mean}
Consider the metric $g=dt^{2}+g_{t}$ on the product
$[0,\infty)\times N$, such that the mean curvature of the cross
section $N$ is uniformly bounded from below by a positive constant
$h_{0}$. Then, for any function $f$ we have
\begin{align}\notag
\int |\partial_{t}f|^{2}d\mu_{g}\geq h^{2}_{0}\int
|f|^{2}d\mu_{g}+h_{0}\int_{t=T}|f|^{2}d\mu_{g_{t}}-h_{0}\int_{t=0}|f|^{2}d\mu_{g_{t}}.
\end{align}
\end{lemma}
\begin{proof}
See Lemma 4.1 in \cite{Biquard}.
\end{proof}

By definition of the metrics $\{g_{j}\}$, the mean curvature of the
cross sections $N_{t}$ on the cusps are uniformly bounded from below
independently of $j$. Using Lemma \ref{mean}, we can now derive the
desired uniform Poincar\'e inequality on functions.

\begin{proposition}\label{poincare1}
There exists a positive constant $c$, independent of $j$, such that
\begin{align}\notag
\int_{\overline{M}}|df|^{2}d\mu_{g_{j}}\geq
c\int_{\overline{M}}|f|^{2}d\mu_{g_{j}}
\end{align}
for any function $f$ on $\overline{M}$ such that
$\int_{\overline{M}}f d\mu_{g_{j}}=0$.
\end{proposition}
\begin{proof}
The argument is by contradiction. The details can be found in
\cite{Biquard}, see Corollaire 4.3.
\end{proof}

Next, we have to derive an uniform Poincar\'e inequality for
1-forms. Given a 1-form $\alpha$ the following lemma holds:

\begin{lemma}\label{Rollin}
There exists $T>0$ such that
\begin{align}\notag
\int_{N}|\nabla\alpha|^{2}+\textrm{Ric}^{g_{j}}(\alpha,\alpha)d\mu_{g_{t}}\geq
k_{1}\int_{N}|\nabla_{\partial_{t}}\alpha|^{2}d\mu_{g_{t}}-k_{2}\int_{N}|\alpha|^{2}d\mu_{g_{t}}
\end{align}
for any $t\in[T, T_{j})$, with $k_{1}>0$ and $k_{2}=O(e^{-t})$.
\end{lemma}
\begin{proof}
Let us start by considering a Poncar\'e metric $g$ on $M$ and the
associated sequence of metrics $\{g_{j}\}$. On each cusp, write
$\alpha$ as
\begin{align}\notag
\alpha=fdt+\alpha_{1}
\end{align}
where $i_{\partial_{t}}\alpha_{1}=0$, with $f$ a real functions. We
then have
\begin{align}\notag
\nabla\alpha=dt\otimes\nabla_{\partial_{t}}\alpha+d^{N}f\otimes
dt-f\textrm{II}_{g_{j}}(\cdot,\cdot)+\textrm{II}_{g_{j}}(\alpha_{1},\cdot)\otimes
dt+\nabla^{N}\alpha_{1}
\end{align}
where by $\textrm{II}$ we denote the second fundamental form of the
slice $N$. By further decomposing $\alpha_{1}$ as
\begin{align}\notag
\alpha_{1}=f_{1}\varphi_{j} d\theta+\alpha_{2}
\end{align}
with $f_{1}$ a real function and where $i_{v_{\theta}}\alpha_{2}=0$,
we can compute all the components of $\nabla\alpha$. Next, one
computes $Ric^{g_{j}}(\alpha, \alpha)$. More precisely, we have
\begin{align}\notag
Ric^{g_{j}}(\alpha, \alpha)=Ric^{g_{\Sigma_{i}}}(\alpha_{2},
\alpha_{2})-\frac{\partial^{2}_{t}\varphi_{j}}{\varphi_{j}}\{|f|^{2}+|f_{1}|^{2}\}+O(e^{-2t})|\alpha|^{2}.
\end{align}
where $Ric^{g_{\Sigma_{i}}}$ is the Ricci curvature of the metric
$g_{\Sigma_{i}}$ coming from \ref{model}. Thus, following a strategy
first outlined by Rollin in \cite{Rollin} and further employed in
\cite{luca1} and \cite{luca3}, the lemma is a consequence of the
Poincar\'e inequality for the circle.
\end{proof}

The idea is now to integrate with respect to the $t$ variable on
each cusp. This will allow us to globalize the slice estimate given
in Lemma \ref{Rollin}. First, observe that for
$[t_{1},t_{2}]\subset[T,T_{j}]$
\begin{align}\notag
\int_{\partial\{[t_{1},t_{2}]\times
N\}}|\alpha|^{2}d\mu_{g_{j}}&=\int_{[t_{1},t_{2}]\times
N}\partial_{t}(|\alpha|^{2}d\mu_{g_{t}})dt=\int_{[t_{1},t_{2}]\times
N}\partial_{t}|\alpha|^{2}d\mu_{g_{t}}dt
\\ \notag
&+\int_{[t_{1},t_{2}]\times
N}|\alpha|^{2}\partial_{t}d\mu_{g_{t}}dt\\ \notag
&=\int_{[t_{1},t_{2}]\times
N}\partial_{t}|\alpha|^{2}d\mu_{g_{j}}-2\int_{[t_{1},t_{2}]\times
N}h|\alpha|^{2}d\mu_{g_{j}}.
\end{align}
We then  obtain
\begin{align}\notag
\int_{[t_{1},t_{2}]\times N}\partial_{t}|\alpha|^{2}d\mu_{g_{j}}\geq
\int_{\partial\{[t_{1},t_{2}]\times
N\}}|\alpha|^{2}d\mu_{g_{j}}+2h_{0}\int_{[t_{1},t_{2}]\times
N}|\alpha|^{2}d\mu_{g_{j}}.
\end{align}
where $h_{0}$ is a uniform lower bound for the mean curvature. But
now
\begin{align}\notag
\partial_{t}|\alpha|^{2}=2(\alpha,\nabla_{\partial_{t}}\alpha)\leq
2|\alpha||\nabla_{\partial_{t}}\alpha|\leq
h_{0}|\alpha|^{2}+\frac{1}{h_{0}}|\nabla_{\partial_{t}}\alpha|^{2}
\end{align}
which then implies
\begin{align}\label{uniform}
\int_{[t_{1},t_{2}]\times
N}|\nabla_{\partial_{t}}\alpha|^{2}d\mu_{g_{j}}\geq
h_{0}\int_{\partial\{[t_{1},t_{2}]\times
N\}}|\alpha|^{2}d\mu_{g_{j}}+h^{2}_{0}\int_{[t_{1},t_{2}]\times
N}|\alpha|^{2}d\mu_{g_{j}}.
\end{align}
We summarize the discussion above into the following lemma.
\begin{lemma}\label{biquard}
There exist positive numbers $c>0$, $T>0$ such that
\begin{align}\notag
\int_{[t_{1},t_{2}]\times
N}|d\alpha|^{2}+|d^{*_{g_{j}}}\alpha|^{2}d\mu_{g_{j}}\geq
c\int_{[t_{1},t_{2}]\times N}|\alpha|^{2}d\mu_{g_{j}}
\end{align}
for any $[t_{1},t_{2}]\subset[T,T_{j})$ and $\alpha$ with support
contained in $[t_{1},t_{2}]\times N$.
\end{lemma}
\begin{proof}
Combining (\ref{uniform}) and Lemma \ref{Rollin} with $T$ big
enough, the result follows from the well know Bochner formula for
1-forms.
\end{proof}

The above lemma is almost the desired uniform Poincar\'e inequality.
To conclude the proof we need few results concerning the convergence
of harmonic 1-forms.

\begin{proposition}\label{1form}
Let $[a]\in H^{1}_{dR}(\overline{M})$ and $\{\alpha_{j}\}$ be the
sequence of harmonic representatives with respect the metrics
$\{g_{j}\}$. Then $\{\alpha_{j}\}$ converges, with respect to the
$C^{\infty}$-topology on compact sets, to a harmonic 1-form
$\alpha\in L^{2}\Omega^{1}_{g}(M)$.
\end{proposition}
\begin{proof}
Let $\beta$ be a closed smooth representative for $[a]\in
H^{1}_{dR}(\overline{M})$. Given $g_{j}$, by the Hodge decomposition
theorem, we can write $\alpha_{j}=\beta+df_{j}$ with $\alpha_{j}$
harmonic and $f_{j}$ a $C^{\infty}$ function. Without loss of
generality we can assume that $\int_{\overline{M}}f_{j}d\mu_{j}=0$.
Furthermore, we have
\begin{align}\notag
0=d^{*}\alpha_{j}=d^{*}\beta+d^{*}df_{j}\Longrightarrow
\Delta^{g_{j}}_{H}f_{j}=-d^{*}\beta
\end{align}
where by $\Delta_{H}$ we denote the Hodge Laplacian. Finally, a
bootstrap argument with elliptic equation above combined with
Proposition \ref{poincare1} gives the convergence result. For more
details see Proposition 4.4. in \cite{Biquard}.
\end{proof}

It is now possible to refine Proposition \ref{1form} and analyze the
convergence in more details. Let $\beta$ be a smooth representative
of the cohomology class $[a]\in H^{1}_{dR}(\overline{M})$. By the
long exact sequence with compact support of the pair $(\overline{M},
\Sigma)$
\begin{align}\notag
...\longrightarrow H^{1}_{c}(M)\longrightarrow
H^{1}(\overline{M})\longrightarrow H^{1}(\Sigma)\longrightarrow...
\end{align}
we can chose $\beta$ as follows
\begin{align}\notag
\beta=\beta_{c}+\sum_{i}\gamma_{i}
\end{align}
where $\beta_{c}$ is a smooth closed 1-form with support not
intersecting the divisor $\Sigma$ and $\gamma_{i}\in
H^{1}(\Sigma_{i};\setR)$ for any $i$. The metric $g$ is
$C^{2}$-asymptotic to a Poincar\'e metric, as a result
\begin{align}\notag
\lim_{t\rightarrow\infty}d^{*_{g}}\gamma_{i}=0\
\end{align}
since $\gamma_{i}$ can be chosen to be harmonic with respect to the
metric $g_{\Sigma_{i}}$ for any choice of the index $i$.
Furthermore, given $\epsilon>0$ we can find $T$ big enough such that
$\lim_{j\rightarrow\infty}\|d^{*_{j}}\gamma_{i}\|_{L^{2}_{g_{j}}(t\geq
T)}\leq\epsilon$. In other words we proved
\begin{lemma}
Given $\epsilon>0$, there exists $T$ big enough such that
\begin{align}\notag
\int_{t\geq T}|d^{*}\beta|^{2}d\mu_{g}\leq\epsilon, \quad
\int_{t\geq T}|d^{*_{j}}\beta|^{2}d\mu_{g_{j}}\leq\epsilon.
\end{align}
\end{lemma}

We can now prove
\begin{lemma}\label{harmonic1}
Given $\epsilon>0$, there exists $T$ big enough such that
\begin{align}\notag
\int_{t\geq T}|\alpha|^{2}d\mu_{g}\leq \epsilon, \quad \int_{t\geq
T}|\alpha_{j}|^{2}d\mu_{g_{j}}\leq \epsilon.
\end{align}
\end{lemma}
\begin{proof}
By construction $\alpha_{j}=\beta+df_{j}$, thus
\begin{align}\notag
\int_{t\geq T}|df_{j}|^{2}d\mu_{g_{j}}=\int_{t=
T}f_{j}\wedge*df_{j}-\int_{t\geq T}(d^{*}df_{j},f_{j})d\mu_{g_{j}}.
\end{align}
But now
\begin{align}\notag
d^{*_{j}}\alpha_{j}=d^{*_{j}}\beta+d^{*_{j}}df_{j}=0\Longrightarrow
d^{*_{j}}df_{j}=-d^{*_{j}}\beta,
\end{align}
thus
\begin{align}\notag
\int_{t\geq
T}|df_{j}|^{2}d\mu_{g_{j}}=\int_{t=T}f_{j}\wedge*df_{j}+\int_{t\geq
T}(d^{*}\beta, f_{j})d\mu_{g_{j}}.
\end{align}
By the Cauchy inequality
\begin{align}\notag
\int_{t\geq T}(d^{*}\beta, f_{j})d\mu_{g_{j}}\leq
\|f_{j}\|_{L^{2}_{g_{j}}}\|d^{*_{j}}\beta\|_{L^{2}_{g_{j}}(t\geq T)}
\end{align}
and then this term can be made arbitrarily small. It remains to
study the term $\int_{t=T}f_{j}\wedge*df_{j}$. Recall that
$f_{j}\rightarrow f$ in the $C^{\infty}$ topology on compact sets.
Thus, for a fixed $T$
\begin{align}\notag
\int_{t=T}f_{j}\wedge*df_{j}\rightarrow \int_{t=T}f\wedge*df.
\end{align}
It remains to show that $\int_{t=T}f\wedge*df$ can be made
arbitrarily small by taking $T$ big enough. Define the function
$F(s)=\int_{t=s}f*df$, since $f\in L^{2}_{1}$ we have $F(s)\in
L^{1}(\setR^{+})$ and then we can find a sequence
$\{s_{k}\}\rightarrow \infty$ such that $F(s_{k})\rightarrow 0$.

\end{proof}
\begin{proposition}\label{biquard1}
There exists $c>0$ independent of $j$ such that
\begin{align}\notag
\int_{\overline{M}}|d\alpha|^{2}+|d^{*_{g_{j}}}\alpha|^{2}d\mu_{g_{j}}\geq
c\int_{\overline{M}}|\alpha|^{2}d\mu_{g_{j}}
\end{align}
for any $\alpha\perp\mathcal{H}^{1}_{g_{j}}$.
\end{proposition}
\begin{proof}
Let us proceed by contradiction. Assume the existence of a sequence
$\{\alpha_{j}\}\in(\mathcal{H}^{1}_{g_{j}})^{\perp}$ such that
$\|\alpha_{j}\|_{L^{2}(g_{j})}=1$ and for which
\begin{align}\notag
\int_{\overline{M}}|d\alpha_{j}|^{2}+|d^{*_{g_{j}}}\alpha_{j}|^{2}d\mu_{g_{j}}\longrightarrow
0
\end{align}
as $j\rightarrow\infty$. By eventually passing to a subsequence, a
diagonal argument shows that $\{\alpha_{j}\}$ converges, with
respect to the $C^{\infty}$-topology on compact sets, to a $1$-form
$\alpha\in L^{2}\Omega^{1}_{g}(M)$. By construction
$\alpha\in\mathcal{H}^{1}_{g}(M)$. On the other hand, Lemma
\ref{harmonic1} combined with the isomorphism provided by Theorem
\ref{Zucker} gives that $\alpha\in(\mathcal{H}^{1}_{g})^{\perp}$. We
conclude that $\alpha=0$. Lemma \ref{biquard} can now be easily
applied to derive a contradiction.
\end{proof}

\section{Convergence of 2-forms}\label{analisi2}

In this section we have to study the convergence of 2-forms. The
first result is completely analogous to the case of 1-forms.

\begin{proposition}\label{2form}
Let $[a]\in H^{2}_{dR}(\overline{M})$ and $\{\alpha_{j}\}$ be the
sequence of harmonic representatives with respect the sequence of
metrics $\{g_{j}\}$. Then $\{\alpha_{j}\}$ converges, with respect
to the $C^{\infty}$-topology on compact sets, to a harmonic
$2$-forms $\alpha\in L^{2}\Omega^{2}_{g}(M)$.
\end{proposition}
\begin{proof}
Given an element $a\in H^{2}_{dR}(\overline{M})$, take a smooth
representative of the form $\beta=\beta_{c}+\sum_{i}\gamma_{i}$
where $\beta_{c}$ is a closed 2-form with support not intersecting
$\Sigma$ and $\gamma_{i}\in H^{2}(\Sigma_{i}; \setR)$ for any $i$.
Given $g_{j}$, let $\alpha_{j}$ be the harmonic representative of
the cohomology class determined by $a$. By the Hodge decomposition
theorem we can write $\alpha_{j}=\beta+d\sigma_{j}$ with
$\sigma_{j}\in (\mathcal{H}^{1}_{g_{j}})^{\perp}$ such that
$d^{*_{j}}\sigma_{j}=0$. Thus
\begin{align}\notag
0=d^{*_{j}}\beta+d^{*_{j}}d\sigma_{j}\Longrightarrow
d^{*}d\sigma_{j}=-d^{*_{j}}\beta.
\end{align}
Taking the global $L^{2}$ inner product of $d^{*}d\sigma_{j}$ with
$\sigma_{j}$ we obtain the estimate
\begin{align}\label{1}
(d^{*}d\sigma_{j},\sigma_{j})_{L^{2}(g_{j})}&=
\|d\sigma_{j}\|^{2}_{L^{2}}=-\int_{\overline{M}}(\sigma_{j},d^{*}\beta)d\mu_{g_{j}}\\
\notag &\leq
\|\sigma_{j}\|_{L^{2}(g_{j})}\|d^{*}\beta\|_{L^{2}(g_{j})}.
\end{align}
By Proposition \ref{biquard1}, we conclude that
\begin{align}\label{2}
\|\sigma_{j}\|^{2}_{L^{2}(g_{j})}\leq
c\|d\sigma_{j}\|^{2}_{L^{2}(g_{j})}.
\end{align}
Combining \ref{1} and \ref{2} we then obtain
\begin{align}\notag
\|\sigma_{j}\|^{2}_{L^{2}(g_{j})}\leq
c\|d\sigma_{j}\|^{2}_{L^{2}(g_{j})}\leq
c\|\sigma_{j}\|_{L^{2}(g_{j})}\|d^{*_{j}}\beta\|_{L^{2}(g_{j})}.
\end{align}
Since $\|d^{*_{j}}\beta\|_{L^{2}(g_{j})}$ is bounded independently
of $j$, we conclude that the same is true for
$\|\sigma_{j}\|_{L^{2}(g_{j})}$ and
$\|d\sigma_{j}\|_{L^{2}(g_{j})}$. We conclude that
$\|\sigma_{j}\|_{L^{2}_{1}(g_{j})}$ is uniformly bounded. Now a
standard diagonal argument allows us to conclude that, up to a
subsequence, $\{\sigma_{j}\}$ weakly converges to an element
$\sigma\in L^{2}_{1}$. Using the elliptic equation
\begin{align}\notag
\Delta^{g_{j}}_{H}\sigma_{j}=-d^{*_{j}}\beta
\end{align}
and a bootstrapping argument it is possible to show that
$\sigma_{j}\rightarrow \sigma$ in the $C^{\infty}$ topology on
compact sets. This proves the proposition.
\end{proof}

We know want to obtain a refinement of Proposition \ref{2form}. We
begin with the following simple lemma.

\begin{lemma}\label{2form1}
Given $\epsilon>0$, there exists $T$ big enough such that
\begin{align}\notag
\int_{t\geq T}|d^{*_{g}}\beta|^{2}d\mu_{g}\leq \epsilon, \quad
\int_{t\geq T}|d^{*_{j}}\beta|^{2}d\mu_{g_{j}}\leq\epsilon.
\end{align}
\end{lemma}
\begin{proof}
Since $\beta=\beta_{c}+\gamma$ with $\gamma$ a fixed element in
$H^{2}(\Sigma;\setR)$, the lemma follows from the definition of the
metrics $\{g_{j}\}$.
\end{proof}

An analogous result holds for the 2-forms $\{d\sigma_{j}\}$.

\begin{lemma}
Given $\epsilon>0$, there exists $T$ big enough such that
\begin{align}\notag
\int_{t\geq T}|d\sigma|^{2}d\mu_{g}\leq \epsilon, \quad \int_{t\geq
T}|d\sigma_{j}|^{2}d\mu_{g_{j}}\leq\epsilon.
\end{align}
\end{lemma}
\begin{proof}
The first inequality follows easily from the fact that $\alpha\in
L^{2}\Omega^{2}_{g}(M)$. By Lemma \ref{2form1}, given $\epsilon>0$
we can find $T$ such that
\begin{align}\notag
\|\sigma_{j}\|_{L^{2}(g_{j})}\Big\{\int_{t\geq
T}|d^{*}\beta|^{2}d\mu_{g_{j}}\Big\}^{\frac{1}{2}}\leq
\frac{\epsilon}{2}
\end{align}
independently of the index $j$. Now
\begin{align}\notag
\int_{t\geq
T}|d\sigma_{j}|^{2}d\mu_{g_{j}}=\int_{t=T}\sigma_{j}\wedge*_{j}d\sigma_{j}-\int_{t\geq
T}(d^{*_{j}}d\sigma_{j},\sigma_{j})d\mu_{g_{j}}
\end{align}
but $d^{*_{j}}d\sigma_{j}=-d^{*_{j}}\beta$, thus
\begin{align}\notag
\int_{t\geq T}|d\sigma_{j}|^{2}d\mu_{g_{j}}\leq
\frac{\epsilon}{2}+\Big\vert\int_{t=T}\sigma_{j}\wedge*_{j}d\sigma_{j}\Big\vert.
\end{align}
Since $\sigma_{j}\rightarrow\sigma$ in the $C^{\infty}$ topology on
compact sets, we have that
$\int_{t=T}\sigma_{j}\wedge*_{j}d\sigma_{j}\rightarrow\int_{t=T}\sigma\wedge*_{g}d\sigma$.
But now $\sigma\in L^{2}_{1}(g)$ and therefore we can conclude the
proof of the proposition.
\end{proof}

\begin{lemma}
$\sigma$ is orthogonal to the harmonic 1-form on $(M,g)$.
\end{lemma}

\begin{proof}
By construction we have $\sigma_{j}\in
(\mathcal{H}^{1}_{g_{j}})^{\perp}$. Recall that fixed a cohomology
element $[a]\in H^{1}_{dR}(\overline{M})$, denoted by
$\{\gamma_{j}\}$ the sequence of the harmonic representatives with
respect to the $\{g_{j}\}$, given $\epsilon>0$ we can chose $T$ such
that $\int_{t\geq T}|\gamma_{j}|^{2}d\mu_{g_{j}}\leq \epsilon$. Now,
given $\gamma\in \mathcal{H}^{1}_{g}$ we want to show that
$(\sigma,\gamma)_{L^{2}(g)}=0$. Since
$H^{1}_{dR}(\overline{M})=\mathcal{H}^{1}_{g}(M)$, we can find a
sequence of harmonic 1-forms $\{\gamma_{j}\}$ such that
$\gamma_{j}\rightarrow\gamma$ in the $C^{\infty}$ topology on
compact sets. Let $K$ be a compact set in $M$, then
\begin{align}
\Big\vert\int_{\overline{M}\backslash
K}(\sigma_{j},\gamma_{j})d\mu_{j}\Big\vert\leq\|\sigma_{j}\|_{L^{2}_{g_{j}}}\|\gamma_{j}\|_{L^{2}_{g_{j}}(\overline{M}\backslash
K)}
\end{align}
can be made arbitrarily small by choosing the compact $K$ big
enough. Since
$(\sigma_{j},\gamma_{j})_{L^{2}(\overline{M},g_{j})}=0$, we have
\begin{align}\notag
\int_{K}(\sigma_{j},\gamma_{j})d\mu_{g_{j}}=-\int_{\overline{M}\backslash
K}(\sigma_{j},\gamma_{j})d\mu_{g_{j}}
\end{align}
and then the integral $\int_{K}(\sigma_{j},\gamma_{j})d\mu_{g_{j}}$
can be made arbitrarily small. On the other hand
\begin{align}\notag
\Big\vert\int_{M}(\sigma,\gamma)d\mu_{g}\Big\vert\leq
\Big\vert\int_{K}(\sigma,\gamma)d\mu_{g}\Big\vert+\|\sigma\|_{L^{2}(M,g)}\|\gamma\|_{L^{2}_{g}(M\backslash
K)}.
\end{align}
Since $\gamma\in L^{2}\Omega^{1}_{g}(M)$ we conclude that $\sigma\in
(\mathcal{H}^{1}_{g})^{\perp}$.
\end{proof}

We now want to study the intersection form of $(\overline{M},g_{j})$
and eventually show the convergence to the $L^{2}$ intersection form
of $(M,g)$. Recall the isomorphism $H^{2}_{dR}(\overline{M})\simeq
\mathcal{H}^{2}(M)$, moreover given $[a]\in
H^{2}_{dR}(\overline{M})$ we can generate $\{\alpha_{j}\}\in
\mathcal{H}^{2}_{g_{j}}(\overline{M})$ that converges in the
$C^{\infty}$ topology on compact sets to a $\alpha\in
\mathcal{H}^{2}_{g}(M)$. We also have that, fixed a compact set $K$,
$*_{j}=*_{g}$ for $j$ big enough. Since
\begin{align}\notag
\alpha_{j}=\alpha^{+_{j}}_{j}+\alpha^{-_{j}}_{j}=\frac{\alpha_{j}+*_{j}\alpha_{j}}{2}+
\frac{\alpha_{j}-*_{j}\alpha_{j}}{2}\rightarrow
\alpha^{+_{g}}+\alpha^{-_{g}}=\alpha
\end{align}
the claim follows. Let us summarize these results into a
proposition.
\begin{proposition}\label{intersection}
Let $(\overline{M}, g_{j})$ and $(M, g)$ be defined as above. Given
$[a]\in H^{2}_{dR}(\overline{M})$ and denoted by $\{\alpha_{j}\}$
the harmonic representatives with respect to the sequence of metrics
$\{g_{j}\}$, we have
\begin{align}\notag
\|\alpha^{+}_{j}\|_{L^{2}(\overline{M},
g_{j})}\rightarrow\|\alpha^{+}\|_{L^{2}(M, g)},
\quad\|\alpha^{-}_{j}\|_{L^{2}(\overline{M},
g_{j})}\rightarrow\|\alpha^{-}\|_{L^{2}(M, g)}.
\end{align}

\end{proposition}

\section{Biquard's construction}\label{bootstrap}

In this section we show how to construct an irreducible solution of
the Seiberg-Witten equations on $(M,g)$, for \emph{any} metric $g$
which is $C^{2}$-asymptotic to a Poincar\'e type metric at infinity.

Fix a $Spin^{c}$ structure on $\overline{M}$, with determinant line
bundle $L$, and let $g$ be a Poincar\'e metric on
$\overline{M}\backslash\Sigma$. Let $\{g_{j}\}$ be the sequence of
metrics on $\overline{M}$ approximating $(M,g)$ constructed in
Section \ref{compactification}. Let us assume we can find
irreducible solutions $(A_{j},\psi_{j})$ of the perturbed
Seiberg-Witten equations on $(\overline{M},g_{j})$
\begin{equation*}
\begin{cases}
\mathcal{D}_{A_{j}}\psi_{j}=0\\
F^{+}_{A_{j}}+i2\pi\omega^{+}_{j}=q(\psi_{j})
\end{cases}
\end{equation*}
where $\omega_{j}=\frac{i}{2\pi}F_{B_{j}}$ and $B_{j}$ is the
connection on the line bundle $\mathcal{O}_{\overline{M}}(\Sigma)$
given by
\begin{align}\notag
B_{j}=d-\sum_{k}i\chi_{j}(\partial_{t}\varphi_{j})\eta_{k}.
\end{align}
The idea is to show that, up to gauge transformations, the
$(A_{j},\psi_{j})$ converge in the $C^{\infty}$ topology on compact
sets to a solution of the unperturbed Seiberg-Witten equations
\begin{equation*}
\begin{cases}
\mathcal{D}_{A} \psi=0\\
F^{+}_{A}=q(\psi)
\end{cases}
\end{equation*}
on $(M,g)$, where $A=C+a$, with $C$ a fixed smooth connection on
$L\otimes \mathcal{O}(-\Sigma)$, and $a\in
L^{2}_{1}(\Omega^{1}_{g}(M))$ with $d^{*}a=0$.

\begin{lemma}\label{bounded}
We have the decomposition
\begin{align} \notag
&s_{g_{j}}=s^{b}_{g_{j}}-2\chi_{j}\frac{\partial^{2}_{t}\varphi_{j}}{\varphi_{j}}\\
\notag
&F_{B_{j}}=-\sum_{k}i\chi_{j}\frac{\partial^{2}_{t}\varphi_{j}}{\varphi_{j}}dt\wedge\varphi_{j}\eta_{k}+F^{b}_{j}
\end{align}
with $s^{b}_{g_{j}}$ and $F^{b}_{j}$ bounded independently of $j$
\end{lemma}
\begin{proof}
See Proposition \ref{scalar curvature}.
\end{proof}

Since $i2\pi \omega_{j}=-F_{B_{j}}$, we can rewrite the perturbed
Seiberg-Witten equations as follows
\begin{equation}\label{perturbed}
\begin{cases}
\mathcal{D}_{A_{j}} \psi_{j}=0\\
F^{+}_{A_{j}}-F^{+}_{B_{j}}=q(\psi_{j}).
\end{cases}
\end{equation}
Recall that in the case under consideration, the twisted Licherowicz
formula \cite{Lawson} reads as follows
\begin{align}\notag
\mathcal{D}^{2}_{A_{j}}\psi_{j}=\nabla^{*}_{A_{j}}\nabla_{A_{j}}\psi_{j}
+\frac{s_{g_{j}}}{4}\psi_{j}+\frac{1}{2}F^{+}_{A_{j}}\cdot\psi_{j}.
\end{align}
By using the SW equations we have
\begin{align}\notag
0=\nabla^{*}_{A_{j}}\nabla_{A_{j}}\psi_{j}+\frac{s_{g_{j}}}{4}\psi_{j}
+\frac{|\psi_{j}|^{2}}{4}\psi_{j}+\frac{1}{2}F^{+}_{B_{j}}\cdot\psi_{j}.
\end{align}
Keeping into account the decomposition given in Lemma \ref{bounded}
we obtain
\begin{align}\notag
0=\nabla^{*}_{A_{j}}\nabla_{A_{j}}\psi_{j}+P_{j}\psi_{j}+P^{b}_{j}\psi_{j}
+\frac{|\psi_{j}|^{2}}{4}\psi_{j}
\end{align}
where on each cusp
\begin{align}\notag
P_{j}\psi_{j}=-\frac{1}{2}\chi_{j}\frac{\partial^{2}_{t}\varphi_{j}}{\varphi_{j}}\psi_{j}
-\frac{i}{2}\chi_{j}\frac{\partial^{2}_{t}\varphi_{j}}{\varphi_{j}}(dt\wedge\varphi_{j}\eta_{i})^{+}\cdot\psi_{j}
\end{align}
with $P^{b}_{j}$ uniformly bounded in $j$. Now, it can be explicitly
checked that for a metric of the form
$dt^{2}+\varphi^{2}_{j}\eta^{2}_{i}+g_{\Sigma_{i}}$ the self-dual
form $(dt\wedge\varphi_{j}\eta_{i})^{+}$ acts by Clifford
multiplication with eigenvalues $\pm i$. The eigenvalues of the
operator $P_{j}$ are then given by $0$ and
$-\chi_{j}\frac{\partial^{2}_{t}\varphi_{j}}{\varphi_{j}}$.
\begin{lemma}\label{bspin}
There exists a constant $Q>0$ such that
\begin{align}\notag
|\psi_{j}(x)|^{2}\leq Q
\end{align}
for every $j$ and $x\in \overline{M}$.
\end{lemma}
\begin{proof}
Since
\begin{align}\notag
\Delta|\psi_{j}|^{2}+2|\nabla_{A_{j}}\psi_{j}|^{2}=2Re\langle\nabla^{*}_{A_{j}}\nabla_{A_{j}}\psi_{j},\psi_{j}\rangle
\end{align}
if $x_{j}$ is a maximum point for $|\psi_{j}|^{2}$ we have
$Re\langle\nabla^{*}_{A_{j}}\nabla_{A_{j}}\psi_{j},\psi_{j}\rangle\geq
0$. In conclusion
\begin{align}\notag
0\geq Re\langle \{P_{j}+P^{b}_{j}\}\psi_{j},
\psi_{j}\rangle_{x_{j}}+\frac{|\psi_{j}|^{4}_{x_{j}}}{4}.
\end{align}
By construction the operator $P_{j}+P^{b}_{j}$ is uniformly bounded
from below, the proof is then complete.
\end{proof}

Since $F^{+}_{A_{j}}-F^{+}_{B_{j}}=q(\psi_{j})$ and by Lemma
\ref{bspin} the norms of the $\psi_{j}$ are uniformly bounded, a
similar estimate holds for $F^{+}_{A_{j}}-F^{+}_{B_{j}}$.

\begin{lemma}\label{1spinor}
There exists a constant $Q>0$ such that
\begin{align}\notag
\|\nabla_{A_{j}}\psi_{j}\|_{L^{2}(\overline{M},g_{j})}\leq Q
\end{align}
for any $j$.
\end{lemma}
\begin{proof}
We have
\begin{align}\notag
0&=\int_{\overline{M}}Re\langle\nabla^{*}_{A_{j}}\nabla_{A_{j}}\psi_{j},
\psi_{j}\rangle d\mu_{g_{j}}+\int_{\overline{M}}Re\langle
\{P^{b}_{j}+P_{j}\}\psi_{j}, \psi_{j}\rangle d\mu_{g_{j}}\\
\notag &+\frac{1}{2}\int_{\overline{M}}Re\langle
q(\psi_{j})\psi_{j},\psi_{j}\rangle d\mu_{g_{j}}\\ \notag
&=\|\nabla_{A_{j}}\psi_{j}\|^{2}_{L^{2}(\overline{M},
g_{j})}+\int_{\overline{M}}Re\langle \{P^{b}_{j}+P_{j}\}\psi_{j},
\psi_{j}\rangle
d\mu_{g_{j}}+\frac{1}{4}\int_{\overline{M}}|\psi_{j}|^{4}d\mu_{g_{j}}
\end{align}
but now
\begin{align}\notag
\int_{\overline{M}}Re\langle \{P^{b}_{j}+P_{j}\}\psi_{j},
\psi_{j}\rangle d\mu_{g_{j}}\geq
-k\|\psi_{j}\|^{2}_{L^{2}(\overline{M},g_{j})}
\end{align}
which then implies
\begin{align}\notag
\|\nabla_{A_{j}}\psi_{j}\|^{2}_{L^{2}(\overline{M},g_{j})}&\leq k
\|\psi_{j}\|^{2}_{L^{2}(\overline{M},
g_{j})}-\frac{1}{4}\|\psi_{j}\|^{4}_{L^{4}(\overline{M}, g_{j})}\\
\notag &\leq k\|\psi_{j}\|^{2}_{L^{2}(\overline{M}, g_{j})}.
\end{align}
Since by Proposition \ref{volumes} the volumes of the Riemannian
manifolds $(\overline{M},g_{j})$ are uniformly bounded, the lemma
follows from Lemma \ref{bspin}.
\end{proof}

Define $C_{j}=A_{j}-B_{j}$ and let $C$ be a fixed smooth connection
on the line bundle $L\otimes\mathcal{O}(-\Sigma)$. By the Hodge
decomposition theorem we can write
\begin{align}\notag
C_{j}=C+\eta_{j}+\beta_{j}
\end{align}
where $\eta_{j}$ is $g_{j}$-harmonic and
$\beta_{j}\in(\mathcal{H}^{1}_{g_{j}})^{\perp}$. Thus
\begin{align}\notag
F^{+}_{C_{j}}=q(\psi_{j})=F^{+}_{C}+d^{+}\beta_{j}.
\end{align}
Since $C$ is a fixed connection 1-form,
$\|F_{C}\|_{L^{2}(\overline{M}, g_{j})}$ is uniformly bounded in the
index $j$. As a result, there exists $Q>0$ such that
\begin{align}\notag
\|d^{+}\beta_{j}\|_{L^{2}(\overline{M},g_{j})}\leq Q
\end{align}
for any $j$. By the Stokes' theorem
\begin{align}\notag
\|d^{+}\beta_{j}\|^{2}_{L^{2}(\overline{M},g_{j})}=\|d^{-}\beta_{j}\|^{2}_{L^{2}(\overline{M},g_{j})}
\end{align}
and we then obtain an uniform upper bound on
$\|d\beta_{j}\|_{L^{2}(\overline{M},g_{j})}$. By gauge fixing, see
for example Section 5.3. in \cite{Morgan1}, we can always assume
$d^{*}\beta_{j}=0$ and $\alpha_{j}$ bounded in $H^{1}(\overline{M},
\setR)$. The Poincar\'e inequality given in Proposition
\ref{biquard1} can then be used to conclude that
\begin{align}\notag
c\|\beta_{j}\|^{2}_{L^{2}(\overline{M}, g_{j})}\leq
\|d\beta_{j}\|^{2}_{L^{2}(\overline{M},g_{j})}\leq 2Q.
\end{align}
By a diagonal argument we can now extract a limit
$\beta_{j}\rightarrow \beta$ with $\beta\in L^{2}_{1}(M,g)$.
Similarly we extract a limit $\eta_{j}\rightarrow \eta$ with $\eta
\in L^{2}(M,g)$ and harmonic with respect to $g$, see Proposition
\ref{1form}.

Define $a_{j}=\eta_{j}+\beta_{j}$ that by construction satisfies
$d^{*}a_{j}=0$. If we fix a compact set $K\subset M$, there exists
$j_{0}$ such that for any $j\geq j_{0}$ the connection $B_{j}$
restricted to $K$ is trivial. Thus, for any $j\geq j_{0}$ we have
$A_{j}=C_{j}$ and then $C=A_{j}-a_{j}$. We know that $a_{j}$ is
uniformly bounded in $L^{2}(\overline{M},g_{j})$, by using Lemma
\ref{1spinor} we conclude that $\|\nabla_{C}\psi_{j}\|^{2}_{L^{2}(K,
g_{j})}$ is bounded independently of $j$. On this compact set $K$ we
can therefore extract a weak limit of the sequence
$\{\psi_{j}\}\rightharpoonup\psi$. By using a diagonal argument and
recalling that in a Hilbert space the norm is lower semicontinuous
with respect the weak convergence, we obtain a limit $\psi \in
L^{2}_{1}(M,g)$. Now, a bootstrap argument based on the
\emph{ellipticity} of the Seiberg-Witten equations can be used to
conclude that the $\psi_{j}$ are indeed smooth and that they
converge, in the $C^{\infty}$-topology on compact sets,  to $\psi$.
By Lemma \ref{bspin}, $\psi$ is uniformly bounded over $M$.

Let us summarize the discussion above into a theorem.

\begin{main}\label{convergence}
Fix a \emph{$Spin^{c}$} structure on $\overline{M}$ with determinant
line bundle $L$. Let $g$ be a metric on $M$ asymptotic to a
Poincar\'e metric in the $C^{2}$-topology, and let $\{g_{j}\}$ the
sequence of metrics on $\overline{M}$ that approximate $g$. Let
$\{(A_{j}, \psi_{j})\}$ be the sequence of solutions of the SW
equations with perturbations $\{F^{+}_{B_{j}}\}$ on
$\{(\overline{M},g_{j})\}$. Then, up to gauge transformations, the
solutions $\{(A_{j},\psi_{j})\}$ converge, in the
$C^{\infty}$-topology on compact sets, to a solution $(A,\psi)$ of
the unperturbed SW equations on $(M,g)$ such that
\begin{center}
\begin{itemize}
\item[-] A=C+a where C is a fixed smooth connection on
$L\otimes\mathcal{O}(-\Sigma)$, $d^{*}a=0$
and $a\in L^{2}_{1}(\Omega^{1}_{g}(M))$;\\
\item[-] $\psi\in L^{2}_{1}(M,g)$ and there exists $Q>0$ such that $sup_{x\in M}|\psi(x)|\leq Q$.
\end{itemize}
\end{center}
\end{main}

It remains to show that Theorem \ref{convergence} can be
successfully applied in the case of a minimal pair $(\overline{M},
\Sigma)$ of logarithmic general type. Furthermore, we have to prove
the solution $(A, \psi)$ so constructed is irreducible.

Recall that by construction $\overline{M}$ is a K\"ahler surface.
Let us consider the standard \emph{$Spin^{c}$} structure on
$\overline{M}$ associated to the complex structure $J$. Let $\omega$
be a K\"ahler metric compatible with $J$, then on $(\overline{M},
\omega, J)$ it is easy to construct an irreducible solution of the
perturbed SW equations. More precisely, $\overline{\psi}=(1, 0)\in
\Omega^{0, 0}\oplus\Omega^{0, 2}$ and the Chern connection
$\overline{A}$ on $K^{-1}_{\overline{M}}$ are solution of
\begin{equation*}
\begin{cases}
\mathcal{D}_{\overline{A}} \overline{\psi}=0\\
F^{+}_{\overline{A}}+i(\frac{s_{\omega}+1}{4})\omega=q(\overline{\psi})
\end{cases}
\end{equation*}
where by $s_{\omega}$ we denote the scalar curvature of the
Riemannian metric associated to $\omega$. Now, if
$b^{+}_{2}(\overline{M})>1$ a cobordism argument \cite{Morgan1}
allows us to conclude that the solutions $(A_{j}, \psi_{j})$ of
\ref{perturbed} are indeed irreducible. On the other hand, when
$b^{+}_{2}(\overline{M})=1$ we have to check that the cohomology
classes represented by $\omega$ and $c_{1}(\mathcal{L})$ are in the
same chamber. First, recall that the logarithmic canonical bundle
$\mathcal{L}$ associated to $(\overline{M}, \Sigma)$ has positive
self-intersection, see Proposition \ref{exceptional}. Moreover,
since the Kodaira dimension of $\mathcal{L}$ is non-negative we have
that for some integer $m$ the divisor associated to $m\mathcal{L}$
is effective. We conclude that $\omega\cdot\mathcal{L}>0$. It then
follows that
\begin{align}\notag
(t\omega+(1-t)\mathcal{L})^{2}>0
\end{align}
for any $t\in [0, 1]$. The $(A_{j}, \psi_{j})$ are then irreducible.
Finally, combining Proposition \ref{exceptional}, Lemma
\ref{lebruno} and Proposition \ref{intersection}, we get that the
solution $(A, \psi)$ given in Theorem \ref{convergence} is
irreducible.

\section{Applications}\label{application}

In a decade long effort, Claude LeBrun has brought to light a
beautiful and deep connection between Seiberg-Witten theory and the
Riemannian geometry of closed 4-manifolds; see for example
\cite{LeBrun5}, \cite{LeBrun7} and the bibliography therein. Many of
these results still represent the cutting edge of our current
understanding of Riemannian geometry in real dimension four.

In this section, we show how the analytical results previously
derived in this work are robust enough to prove suitable
generalizations of many of LeBrun's results.

Let us begin by reviewing Chern-Weil theory in the Poincar\'e
metric. Recall that for a compact oriented 4-manifold $N$, the
Gauss-Bonnet and Hirzebruch theorems state that
\begin{align}\notag
\chi(N)=\int_{N}E(g)d\mu_{g}, \quad \sigma(N)=\int_{N}L(g)d\mu_{g}
\end{align}
where $E(g)$ and $L(g)$ are respectively the Euler and signature
forms associated to the metric $g$.

For non-compact manifolds the above curvature integrals might be not
defined or dependent on the choice of the metric. Nevertheless, if
the manifold has finite volume and bounded curvature these curvature
integrals are defined. In this case it remains to study their metric
dependence. Here, we want to compute
\begin{align}\notag
\chi(M,g)=\int_{M}E(g)d\mu_{g}, \quad
\sigma(M,g)=\int_{M}L(g)d\mu_{g}
\end{align}
when $M$ is obtained from a pair $(\overline{M}, \Sigma)$ of
logarithmic general type and $g$ is $C^{2}$-asymptotic to a
Poincar\'e metric at infinity. The following proposition computes
the characteristic numbers of $(M,g)$ in terms of
$\chi(\overline{M})$, $\sigma(\overline{M})$, and a contribution
coming from the cusp ends of $M$.
\begin{proposition}\label{gaussbonnet1}
Let $M$ be equipped with a metric $g$ asymptotic in the
$C^{2}$-topology to a Poincar\'e metric. Then, we have the
equalities
\begin{align}\notag
\chi(M,g)=\chi(\overline{M})-\chi(\Sigma)=\chi(M), \quad
\sigma(M,g)=\sigma(\overline{M})-\frac{1}{3}\Sigma^{2},
\end{align}
where by $\Sigma^{2}$ we indicate the self-intersection of the
boundary divisor.
\end{proposition}
\begin{proof}
A proof can be given generalizing a computation of Biquard, see
Proposition 3.4. in \cite{Biquard}. Alternatively, one can apply a
very general result of M. Stern, see Theorem 1.1 in \cite{Stern}.
\end{proof}

We can now study the Riemannian functional
$\int_{M}s^{2}_{g}d\mu_{g}$ restricted to the space of metrics with
Poincar\'e type asymptotic.

\begin{main}\label{minima}
Let $(\overline{M}, \Sigma)$ be a minimal pair of logarithmic
general type. Assume the boundary divisor $\Sigma$ to be smooth and
without rational components. Let $M$ be equipped with a metric $g$
asymptotic to a Poincar\'e metric in the $C^{2}$-topology. Then
\begin{align}\notag
\frac{1}{32\pi^{2}}\int_{M}s^{2}_{g}d\mu_{g}\geq
(c^{+}_{1}(\mathcal{L}))^{2}
\end{align}
with equality if and only if $g$ has constant negative scalar
curvature, and is K\"ahler with respect to a complex structure
compatible with the standard \emph{$Spin^{c}$} structure on $M$.
\end{main}
\begin{proof}
The first step is to apply Theorem \ref{convergence} with respect to
the standard \emph{$Spin^{c}$} structure on $\overline{M}$. The
solution so obtained $(A, \psi)$ is then irreducible and $\psi\in
L^{2}_{1}$. Finally, a Bochner type argument as in Theorem
\ref{bochner} concludes the proof.
\end{proof}

Regarding the equality case in Theorem \ref{minima}, H. Auvray has
recently proved a \emph{uniqueness} result, see \cite{Auvray}. More
precisely, he is able to show the uniqueness of K\"ahler metrics, in
arbitrary classes, with constant scalar curvature and Poincar\'e
type asymptotic. Up to now, the result is proved for pairs
$(\overline{M}, \Sigma)$ with \emph{ample} log-canonical bundle. In
the case of log-surfaces with smooth boundary, the ampleness of the
log-canonical bundle can be neatly characterized, see Proposition
\ref{ample}.

Finally, Theorem \ref{minima} can be used to prove the following:

\begin{corollary}
Let $(\overline{M}, \Sigma)$ be as above. Then
\begin{align}\notag
\frac{1}{32\pi^{2}}\int_{M}s^{2}_{g}d\mu_{g}\geq
(K_{\overline{M}}+\Sigma)^{2}
\end{align}
with equality iff the metric $g$ is K\"ahler-Einstein with respect
to a complex structure compatible with the standard
\emph{$Spin^{c}$} structure on $M$.
\end{corollary}

In the last thirty years, the problem of constructing
K\"ahler-Einstein metrics for pairs $(\overline{M}, \Sigma)$ of
log-general type was addressed by several authors, see for example
\cite{Cheng}, \cite{kobayashi} and \cite{Tian}. In \cite{Cheng},
\cite{kobayashi}, a unique K\"ahler-Einstein metric is constructed
for any pair $(\overline{M}, \Sigma)$ with ample log-canonical
bundle. This metric is \emph{quasi-isometric} to a Poincar\'e type
metric at infinity.

A more general existence and uniqueness theorem is proved by
Tian-Yau in \cite{Tian}. Nevertheless, the asymptotic behavior of
the K\"ahler-Einstein metric so constructed is much more complicated
to understand. In fact, it must be quite different from a Poincar\'e
type metric as one can check on explicit examples coming from the
theory of locally symmetric varieties.

We conclude this section by presenting an obstruction for Einstein
metrics on blow-ups.

\begin{main} \label{ostruzione}
Let $(M,g)$ as above. Let $M^{'}$be obtained from $M$ by blowing up
$k$ points. If $k\geq\frac{2}{3}(K_{\overline{M}}+\Sigma)^{2}$, then
$M^{'}$ does not admit a Poincar\'e type Einstein metric.
\end{main}
\begin{proof}
By a result of Morgan-Friedman \cite{Morgan3}, we know that the
manifold $\overline{M}\sharp k\overline{\setC P^{2}}$ admits at
least $2^{k}$ different \emph{$Spin^{c}$} structures with
determinant line bundles
\begin{align}\notag
L=K^{-1}_{\overline{M}}\pm E_{1}\pm...\pm E_{k}
\end{align}
for which the SW equations have irreducible solutions for each
metric. Since
\begin{align}\notag
(c_{1}(L)^{+})^{2}&=(c_{1}(\overline{M})^{+}\pm E^{+}_{1}\pm...\pm
E^{+}_{k})^{2}\\ \notag
&=(c_{1}(\overline{M})^{+})^{2}+2\sum_{i}c_{i}(\overline{M})^{+}\cdot\pm
E^{+}_{i}+(\sum_{i}\pm E^{+}_{i})^{2}
\end{align}
we can chose a \emph{$Spin^{c}$} structure whose determinant line
bundle satisfies
\begin{align}\notag
(c_{1}(L)^{+})^{2}\geq (c_{1}(\overline{M})^{+})^{2}\geq
c_{1}(\overline{M})^{2}=c^{2}_{1}(\overline{M}).
\end{align}
We can now apply Theorem \ref{convergence} for any of the
\emph{$Spin^{c}$} structure above and with respect to the metric $g$
on $M^{'}$. We then construct $2^{k}$ irreducible solutions
$(A,\psi)\in L^{2}_{1}(M^{'},g)$, where $A=C+a$ with $C$ a fixed
smooth connection on $L\otimes\mathcal{O}(-\Sigma)$ and $a\in
L^{2}_{1}(\Omega^{1}_{g}(M^{'}))$. By appropriately choosing the
\emph{$Spin^{c}$} structure and using Theorem \ref{bochner} we
compute
\begin{align}\notag
\frac{1}{32\pi^{2}}\int_{M^{'}}s^{2}d\mu_{g}&\geq
(c_{1}(L\otimes\mathcal{O}(-\Sigma))^{+})^{2}
\\ \notag &
\geq (K_{\overline{M}}+\Sigma)^{2}.
\end{align}
By Proposition \ref{gaussbonnet1} one has
\begin{align}\notag
&\chi(M^{'},g)=\chi(\overline{M})-\chi(\Sigma)+k\\ \notag
&\sigma(M^{'},g)=\sigma(\overline{M})-\frac{1}{3}\Sigma^{2}-k.
\end{align}
Since $\chi(\Sigma)=-K_{\overline{M}}\cdot\Sigma-\Sigma^{2}$, if we
assume $g$ to be Einstein
\begin{align}\notag
(K_{\overline{M}}+\Sigma)^{2}-k&=2\chi(M^{'})+3\sigma(M^{'})
\\ \notag &=\frac{1}{4\pi^{2}}\int_{M^{'}}2|W_{+}|^{2}+\frac{s^{2}}{24}d\mu_{g}\\
\notag &\geq\frac{1}{96\pi^{2}}\int_{M^{'}}s^{2}d\mu_{g}\\
\notag &\geq\frac{1}{3}(K_{\overline{M}}+\Sigma)^{2}
\end{align}
so that
\begin{align}\notag
\frac{2}{3}(K_{\overline{M}}+\Sigma)^{2}\geq k.
\end{align}
In other words if
\begin{align}\notag
k>\frac{2}{3}(K_{\overline{M}}+\Sigma)^{2}
\end{align}
we cannot have a Poincar\'e type Einstein metric on $M\sharp
k\overline{\setC P^{2}}$. The equality case can also be included and
the proof goes as in the compact case. For more details, see
\cite{LeBrun5}. The proof is then complete.
\end{proof}

\noindent\textbf{Acknowledgements}. I would like to thank Professor
Claude LeBrun for several enlightening math discussions during my
years st Stony Brook. Moreover, I would like to thank Professor Mark
Stern for explaining to me the proof of Theorem \ref{Zucker} and for
informing me about the results contained in \cite{Stern}. I also
would like to thank Professor Mark Stern for constructive comments
on the paper.

\address{Mathematics Department, Duke University, Box 90320, Durham, NC 27708,
USA}

\emph{E-mail address}: \email{luca@math.duke.edu}

\end{document}